\documentclass{article}
\usepackage{amssymb,amsmath,amsthm,graphicx}
\usepackage{array}
\usepackage{longtable}
\usepackage{makecell}

\def\qed{\hfill {\hbox{${\vcenter{\vbox{               
   \hrule height 0.4pt\hbox{\vrule width 0.4pt height 6pt
   \kern5pt\vrule width 0.4pt}\hrule height 0.4pt}}}$}}}

\def\utr{\, \underline{\triangleright}\, }
\def\otr{\, \overline{\triangleright}\, }

\newtheorem{theorem}{Theorem}

\newtheorem{proposition}[theorem]{Proposition}

\theoremstyle{definition}
\newtheorem{example}{Example}
\newtheorem{definition}{Definition}

\date{}

\title{\Large \textbf{Biquandle Arrow Weight Quiver Representations}}

\author{Sam Nelson\footnote{Email: Sam.Nelson@cmc.edu. Partially supported 
by Simons Foundation collaboration grant 702597.}\and
Migiwa Sakurai\footnote{Email: migiwa@shibaura-it.ac.jp}}

\begin{document}
\maketitle

\begin{abstract}
We define an infinite family of quiver representation-valued invariants
of classical and virtual knots associated to a choice of data vector
consisting of a biquandle, abelian group, set of biquandle arrows weights
with values in the abelian group, coefficient ring and set of biquandle 
endomorphisms. As an application we extract four new polynomial invariants
as decategorifications. We provide examples to show that these invariants
are proper enhancements of the biquandle counting invariant and biquandle 
coloring quiver.
\end{abstract}

\parbox{6in} {\textsc{Keywords:} Biquandles, homsets, enhancements, virtual 
knots, Gauss diagrams, biquandle arrow weights, quiver representations, 
cagtegorification

\smallskip

}

\section{\large\textbf{Introduction}}\label{I}

\textit{Biquandles} are objects in an algebraic category designed for 
defining computable invariants of oriented classical and virtual knots. 
More precisely, to every oriented classical or virtual
knot $K$ there is an associated \textit{fundamental biquandle}
$\mathcal{B}(K)$; given any finite biquandle $X$, the \textit{homset} 
or set of biquandle homomorphisms $\mathrm{Hom}(\mathcal{B}(K),X)$ is
an invariant representable by a set of \textit{colorings} of a diagram
of $K$ by elements of $X$. These colorings of knot diagrams can also be 
represented as colorings of corresponding \textit{Gauss diagrams} 
representing the preimage $S^1$ of the knot with signed arrows
representing crossing points.

Further invariants of oriented classical and virtual knots known as
\textit{enhancements} can be defined via invariants $\phi$ of homset 
elements. More precisely, the multiset of $\phi$-values over the homset 
defines an invariant of classical and virtual knots. More about biquandles
can be found in \cite{EN}.

In \cite{NS} we defined an enhancement of the biquandle homset invariant
using an algebraic structure called a \textit{biquandle arrow weight}, 
a map from pairs of elements of a coloring biquandle $X$ into an abelian
group $A$ such that the sum of such weights over all arrow crossings in 
a biquandle-colored Gauss diagram representing a biquandle homset element
defines an enhancement of the homset invariant.

In \cite{CN} a special case of the biquandle homset invariant was 
categorified to give a quiver-valued invariant of classical and virtual 
knots, the \textit{quandle coloring quiver}. In \cite{NS2} we applied this
construction to categorify the biquandle arrow weight multiset invariant, 
obtaining several new polynomial invariants via decategorification.

In this paper we adapt a construction from \cite{N24} to define a quiver 
representation-valued invariant of oriented classical and virtual knots. 
From this quiver representation we define further polynomial 
decategorification invariants. The paper is organized as follows. In Section 
\ref{R} we review the basics of biquandles, arrow weights and quivers. 
In Section \ref{BW} we introduce our new construction and show that it 
defines an infinite family of invariants of oriented classical and virtual 
knots. We collect some examples and we conclude in 
Section \ref{Q} with some questions for future research.

This paper, including all text, illustrations, and python code for 
computations, was written strictly by the authors without the use of 
generative AI in any form.

\section{\large\textbf{Review of Biquandles, Arrow Weights and Quivers}}\label{R}

In this section, we will review biquandles, arrow weights and quivers. A 
fuller account can be found in our earlier studies \cite{NS} and \cite{NS2}.

A \textit{biquandle} is an algebraic structure on the set 
$X$ with two binary operations $\utr,\otr$ (see \cite{EN} or \cite{NS}), 
subject to axioms derived from the Reidemeister moves.
More precisely, a \textit{biquandle coloring} (or an \textit{$X$-coloring}) 
of an oriented knot, or link diagram $D$ by a biquandle $X$, is a rule that 
assigns to each semiarc in $D$
an element of $X$, such that at every crossing the condition
\[\scalebox{0.58}{\includegraphics{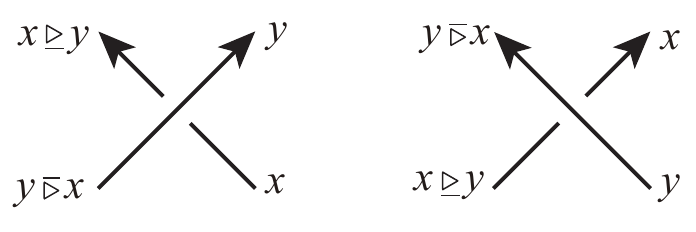}}\]
holds. This condition is called the the \textit{biquandle coloring condition}.

More precisely, a biquandle is a set $X$ with operations $\utr, \otr$
such that 
\begin{itemize}
\item[(i)] for all $x\in X$ we have $x\utr x=x\otr x$,
\item[(ii)] for all $x,y\in X$ the map $S:X\times X\to X\times X$ and the
maps $\alpha_y,\beta_y:X\to X$ given by 
\[S(x,y)=(y\otr x, x\utr y), \quad \alpha_y(x)=x\utr y\quad
\mathrm{and}\quad \beta_y(x)=x\otr y\]
are invertible, and 
\item[(iii)] for all $x,y,z\in X$ we have
\[\begin{array}{rcl}
(x\otr y)\otr(z\otr y) & = & (x\otr z)\otr(y\utr z) \\
(x\otr y)\utr(z\otr y) & = & (x\utr z)\otr(y\utr z) \\
(x\utr y)\utr(z\otr y) & = & (x\utr z)\utr(y\utr z) \\
\end{array}.\]
\end{itemize} 

Finite biquandles can be specified by listing their operation tables; for
example, the tables
\[
\begin{array}{r|rrr}
\utr & 1 & 2 & 3 \\ \hline
1 & 1 & 1 & 1 \\
2 & 2 & 2 & 3 \\
3 & 3 & 3 & 3 \\
\end{array}
\quad
\begin{array}{r|rrr}
\utr & 1 & 2 & 3 \\ \hline
1 & 1 & 3 & 2 \\
2 & 3 & 2 & 1 \\
3 & 3 & 1 & 3 \\
\end{array}
\]
specify a biquandle structure on the set $X=\{1,2,3\}$. See \cite{EN} for more.

The biquandle axioms represent the necessary conditions for biquandle
colorings to be preserved by the Reidemeister moves. They state that, for 
any biquandle coloring of a diagram before a Reidemeister move, there exists 
a unique coloring of the diagram after the move. This coloring coincides with 
the original coloring everywhere except within the neighborhood of the move.
Hence, we see that the number of colorings of an oriented knot or link 
diagram with respect to a finite biquandle $X$ is an integer-valued invariant 
(which may be zero). This integer-valued invariant is denoted by 
$\Phi_X^{\mathbb{Z}}$ and is called the \textit{biquandle counting invariant}.

In particular, for each oriented classical or virtual knot $K$, there is a 
\textit{fundamental biquandle} $\mathcal{B}(K)$ described using generators 
from semiarcs and relations at the crossings. An $X$-coloring gives, and 
is uniquely given by, a \textit{biquandle homomorphism}
$f:\mathcal{B}(K)\to X$, modulo a choice of presentation of the fundamental
biquandle of the knot; see \cite{N24} for an in-depth discussion. 
The sets $\mathrm{Hom}(\mathcal{B}(K),X)$ of biquandle homomorphisms, 
called the \textit{biquandle homset}, form an invariant of $K$ for each
finite biquandle $X$.
More precisely, selecting another diagram for $K$ alters how each homset 
element is represented, with two $X$-colored diagrams representing the 
same homset element related by $X$-colored Reidemeister moves. 

Choosing a diagram for $K$ is similar to choosing the input and output bases 
for vector spaces in linear algebra.
In this comparison, $X$-colored diagrams corresponds to homset elements, 
just as matrices correspond to linear transformations, where homset elements 
plays the role of a change of basis matrices.

Let $K$ be a virtual knot, and let $D$ be a virtual knot diagram of an 
oriented virtual knot $K$.
Then $D$ can be seen as the image $f(\mathbb{S}^1)$ of a generic immersion 
$f : \mathbb{S}^1$ $\to$ $\mathbb{R}^2$.
Recall that a \textit{Gauss diagram} for $D$ is the preimage of $D$ with 
chords, where each one links the preimages of each real crossing.
We can describe the crossing information of each real crossing on the 
corresponding chord by directing the chord toward the undercrossing point 
and assigning to each chord with the sign (that is, the local writhe) of the 
crossing, as shown:
\[\includegraphics[width=4.5cm,clip]{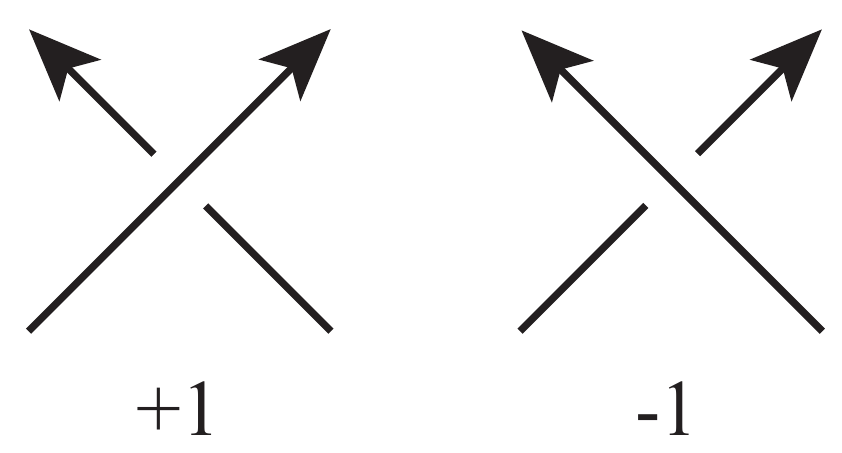}.\]

A \textit{biquandle coloring} of a Gauss diagram by a biquandle $X$ is an 
assignment of an element of $X$ to each segment of the circle between the 
arrowheads and tails, such that the condition 
\[\scalebox{0.7}{\includegraphics{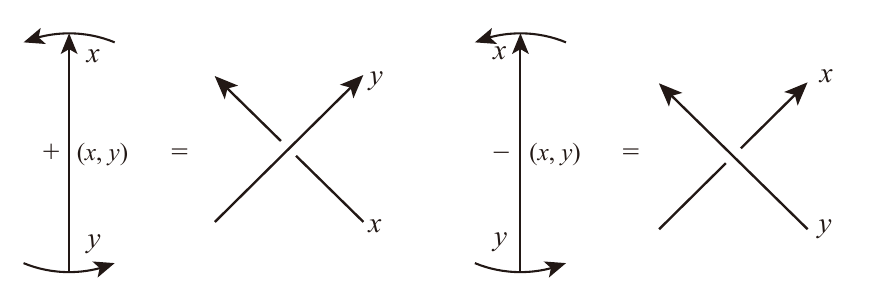}}\]
at every arrow is satisfied.
In particular, to each arrow in a biquandle-colored Gauss diagram we 
associate an ordered pair $(x,y)$ of biquandle elements together with a sign.
The set of biquandle colorings of a Gauss diagram $D$ that represents a 
virtual knot $K$ by a biquandle $X$ can be identified with the homset 
$\mathrm{Hom}(\mathcal{B}(K),X)$, where each $X$-coloring represents a 
homset element, and choosing diagram $D$ is analogous to choosing a basis 
for a vector space.

Next, recall from \cite{NS} that a \textit{biquandle arrow weight} is a 
function $\phi$ that sends pairs of pairs of biquandle elements to elements 
of an abelian group $A$, and that it satisfies axioms found in \cite{NS}.
The biquandle arrow weight axioms are chosen so that the sum $\Sigma_D$ of 
$\phi((a,b),(c,d))$ values over the set of arrow crossings with arrows 
labeled $(a,b)$ and $(c,d)$, and sign given by the product of arrow signs 
over the set of arrows in an $X$-colored Gauss diagram
\[\parbox{15em}{\includegraphics{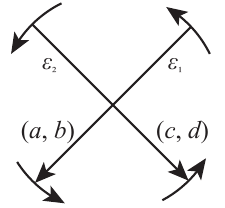}} \quad +\varepsilon_1\varepsilon_2\phi((a,b),(c,d))\]
remain unchanged under Gauss diagram Reidemeister moves.

Given a finite biquandle $X$ and a biquandle arrow weight $W$ with values 
in an abelian group $A$, for any oriented classical or virtual knot $K$ the 
multiset of values over the set of $X$-colorings of any Gauss diagram $D$ 
representing $K$ is an invariant of knots, called the \textit{biquandle 
arrow weight multiset} and denoted by $\Phi_X^{W,M}(K)$.
We can turn this multiset into a polynomial form by summing, over the 
multiset, a formal variable raised to the power of each multiset element, 
which gives the \textit{biquandle arrow weight polynomial} $\Phi_X^{W}(K)$.
Note that, depending on $A$, the polynomial may be seen as a bookkeeping 
device rather than a true polynomial, since the exponents are elements of $A$.

We then have:

\begin{theorem} (\cite{NS})
The biquandle arrow weight multiset and polynomial are invariants of classical
and virtual knots.
\end{theorem}

Finally, recall the infinite family of categorfications of the biquandle
arrow weight invariants introduced in \cite{NS2}.
Let $K$ be an oriented virtual knot represented by a Gauss diagram $D$
(see \cite{EN,GPV}).
Let $X$ be a finite biquandle, $A$ an abelian group, and $\beta$
a biquandle arrow weight for $X$ with coefficients in $A$.
Recall that a map $\phi:X\to X$ is called \textit{biquandle endomorphism} if 
\[\phi(x\, \utr\, x')=\phi(x)\, \utr\, \phi(x') \ \mathrm{and} \ 
\phi(x\, \otr\,  x')=\phi(x)\, \otr\, \phi(x').\]
holds for all $x,x'\in X$.
Given an $X$-coloring of $D$, if we apply $\phi$ to all of the semiarc 
colors
\[\includegraphics{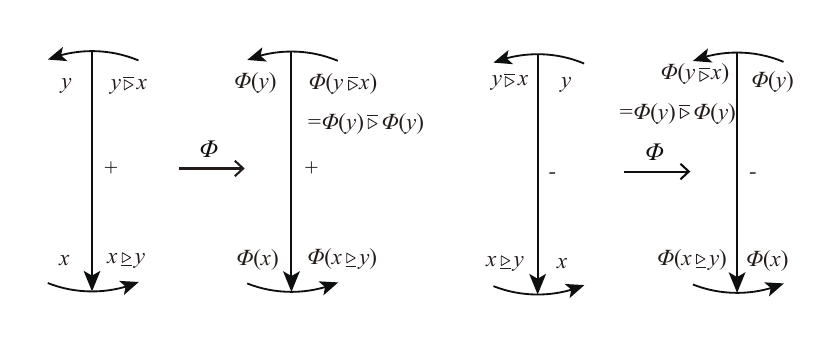}\]
the result is again an $X$-coloring of $D$.

It then follows, as in \cite{CN}, that the directed graph or \textit{quiver}, 
with a vertex for each $X$-coloring of $D$ and directed edge connecting each 
coloring to its image under each $\phi\in S$, is unchanged by Reidemeister 
moves and therefore defines an invariant of virtual knots, which we call the 
\textit{biquandle coloring quiver} $\mathcal{BCQ}_{X,S}(K)$.
When $S=\mathrm{Hom}(X,X)$ is the complete set of endomorphisms, we say 
that $\mathcal{BCQ}_{X,S}(K)$ is the \textit{full quiver}.

Since the vertices in the quiver represent the elements of the biquandle
homset, given a biquandle arrow weight $W$ we can assign to each vertex 
in the quiver the corresponding $\Sigma_D$ value, obtaining a 
\textit{weighted quiver} $\mathcal{Q}_{X,W,S}(K)$, that is, a quiver 
weighted by elements of $A$ at the vertices.
We showed the following theorem in \cite{NS2}:

\begin{theorem}[\cite{NS2}]
The weighted quiver $\mathcal{Q}_{X,W,S}(K)$ is an invariant of oriented 
classical and virtual knots.
\end{theorem}

Quivers are categories with vertices as objects and directed paths as 
morphisms.
Thus, this weighted quiver categorifies the biquandle arrow weight 
polynomial, which can be viewed as a decategorification obtained by summing 
over the vertices terms of the form $u^{w(v)}$, where $w(v)$ is the weight at 
each vertex.

We also obtain two standard polynomial decategorifications, giving 
polynomial invariants of classical and virtual knots.

\begin{definition}[\cite{NS2}]
Let $X$ be a finite biquandle, $S\subset \mathrm{Hom}(X,X)$ a set of 
biquandle endomorphisms, $A$ an abelian group and $W$ a biquandle arrow
weight with $A$ coefficients. Then for any oriented classical or virtual 
knot or link $K$ represented by a Gauss diagram $D$, we define 
\begin{itemize}
\item \textit{The arrow weight in-degree polynomial} to be the sum, over all
vertices $v$ in the vertex set $V$ of the quiver, 
of products of a formal variable $u$ to the power of the vertex weight 
$w(v)=\Sigma_D$ at the vertex $v$ times a formal variable $w$ to the power
of the in-degree of the vertex $v$ (the number of arrows directed in to $v$),  
\[\Phi_{X,W,S}^{\mathrm{deg}_+}(K)=\sum_{v\in V} u^{w(v)}w^{\deg_+(v)}\]
and we further define
\item \textit{The arrow weight two-variable polynomial} to be the
sum, over all edges $e$ in the edge set $E$ of the quiver, of products of 
a formal variable $s$ to the power of the vertex weight at the source 
vertex $S(e)$ times a formal variable $t$ to the power of the vertex weight
at the target vertex $T(e)$,
\[\Phi_{X,W,S}^{2}(K)=\sum_{e\in E} s^{S(e)}t^{T(e)}.\]
\end{itemize}
\end{definition}

\begin{proposition}[\cite{NS2}]
The polynomials $\Phi_{X,W,S}^{\mathrm{deg}_+}(K)$ and $\Phi_{X,W,S}^{2}(K)$ are 
invariants of oriented classical and virtual knots.
\end{proposition}

See \cite{NS2} for more.

\section{\Large\textbf{Biquandle Arrow Weight Quiver Representations}}\label{BW}

We now come to our main new result.

\begin{definition}
Let $X$ be a finite biquandle, $A$ an abelian group (usually finite), $B$ a 
set of biquandle arrow weights with values in $A$, $k$ a choice of coefficient
ring (e.g., $\mathbb{Z}$ or $\mathbb{C}$) and $S$ a set of endomorphisms of
$X$. We will summarize this set of choices as a data vector $(X,A,B,k,S)$.

Given an oriented classical or virtual knot $K$, we then define a quiver 
representation from the biquandle coloring quiver $\mathcal{BCQ}_{X,S}(K)$
to the category of $k$-modules as follows. First, to each vertex in
in $\mathcal{BCQ}_{X,S}(K)$ we assign a copy of $k[A]$, the free $k$-module on
$A$. Each arrow $\mathcal{A}$ in $\mathcal{BCQ}_{X,S}(K)$ connects a vertex 
$v$ representing
an $X$-colored Gauss diagram corresponding to another vertex $\sigma(v)$
representing another $X$-colored diagram for some endomorphism $\sigma\in S$. 
Each arrow weight $b\in B$ evaluates to an element of $A$ on each vertex
in the quiver; then along a fixed arrow each $b\in B$ sends an element 
$b(v)\in A$, i.e. a basis vector for $k[A]$, to another basis vector 
$b(\sigma(v))$. Summing over $b\in B$, we obtain a linear transformation
$f_{\mathcal{A}}:k[A]\to k[A]$ associated to each arrow $\mathcal{A}$. 
This assignment of modules to vertices and linear transformations to edges
constitutes the \textit{biquandle arrow weight quiver representation}
associated to the data vector $\vec{v}=(X,A,B,k,S)$, which we will denote
as $\mathcal{QR}_{\vec{v}}(K)$.
\end{definition}

In particular, at each vertex we have a \textit{distinguished subspace}
of $k[A]$ generated by the elements $b(v)$ for $b\in B$, and the linear
transformations $f_{\mathcal{A}}$ send these distinguished subspaces to
distinguished subspaces.

Analogously to Proposition 3 in \cite{N24}, we have:

\begin{theorem}
For any choice of data vector $(X,A,B,k,S)$, the resulting biquandle arrow
weight quiver representation is an invariant of oriented classical and
virtual knots.
\end{theorem}

\begin{proof}
Nothing in the construction changes under Reidemeister moves.
\end{proof}

Let us illustrate the construction with an example.

%

\begin{example}
Let $K$ be the virtual trefoil knot $2.1$. Let $X$ be the biquandle given 
by the operation tables 
\[
\begin{array}{r|rr}
\utr & 1 & 2 \\ \hline
1 & 2 & 2 \\
2 & 1 & 1
\end{array}
\quad
\begin{array}{r|rr}
\otr & 1 & 2 \\ \hline
1 & 2 & 2 \\
2 & 1 & 1
\end{array},
\] 
let $A=\mathbb{Z}_4=\{\vec{0},\vec{1},\vec{2},\vec{3}\}$
where we write vector arrows to remind ourselves that the elements
of $A$ are basis vectors for $k[A]$, and let $k=\mathbb{Z}$. We compute 
that $X$ has eight biquandle
arrow weights over $A$, including the two represented by the 4-tensors
\[
b_1=
\left[\begin{array}{rr}
\left[\begin{array}{rr} 0 & 3 \\ 1 & 2 \end{array}\right] & 
\left[\begin{array}{rr} 3 & 0 \\ 0 & 3 \end{array}\right] \\
& \\
\left[\begin{array}{rr} 1 & 0 \\ 0 & 1 \end{array}\right] &
\left[\begin{array}{rr} 2 & 3 \\ 1 & 0\end{array}\right] 
\end{array}\right]\ \mathrm{and} \
b_2=
\left[\begin{array}{rr}
\left[\begin{array}{rr} 0 & 1 \\ 1 & 0 \end{array}\right] & 
\left[\begin{array}{rr} 1 & 0 \\ 2 & 3 \end{array}\right] \\
& \\
\left[\begin{array}{rr} 1 & 2 \\ 0 & 3 \end{array}\right] &
\left[\begin{array}{rr} 0 & 3 \\ 3 & 0\end{array}\right] 
\end{array}\right].
\]

The two $X$-colorings of $K$ are represented as Gauss diagrams by
\[\includegraphics{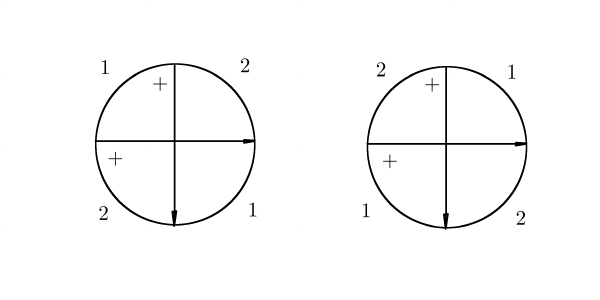}\]
Then the left diagram has arrow weights 
$b_1((2,1),(1,2))=\vec{0}$ and $b_2((2,1),(1,2))=\vec{2}$ 
and the right diagram has arrow weights
$b_1((1,2),(2,1))=\vec{0}$ and $b_2((1,2),(2,1))=\vec{2}$ 
respectively. There are two endomorphisms of $S$, representable
as $S=\{[1,2],[2,1]\}$ where we write $\sigma=[\sigma(1),\dots, \sigma(n)]$.
The arrow from the left coloring to the right induced by $[2,1]$ sends the
$\mathbb{Z}[A]$ basis vector $b_1(v_1)=\vec{0}$ to $b_1(v_2)=\vec{0}$ and
sends $b_2(v_1)=\vec{2}$ to $b_2(v_2)=\vec{2}$; summing these, we get 
the matrix
\[\left[\begin{array}{rrrr} 
1 & 0 & 0 & 0 \\
0 & 0 & 0 & 0 \\
0 & 0 & 1 & 0 \\
0 & 0 & 0 & 0 \\
\end{array}\right]
\]
along this arrow. Repeating for the other arrows, we have invariant
quiver representation $\mathcal{QR}_{\vec{v}}(K)$ (where 
$\vec{v}=(X,\mathbb{Z}_4,\{b_1,b_2\},\mathbb{Z},\{[1,2],[2,1]\})$)
given by
\[\includegraphics{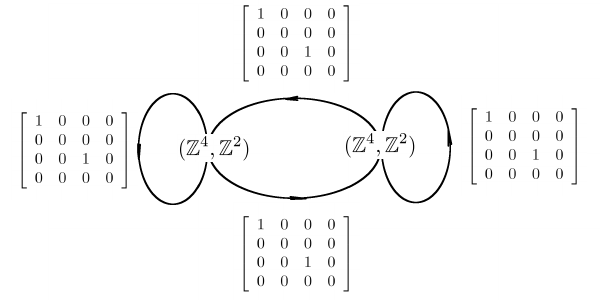}.\]
\end{example}

\begin{example}\label{ex:2}
Let $X$ be the biquandle given by the operation tables
\[
\begin{array}{r|rrr}
\utr & 1 & 2 & 3 \\ \hline
1 & 3 & 3 & 3 \\
2 & 1 & 1 & 1 \\
3 & 2 & 2 & 2 \\
\end{array}
\quad
\begin{array}{r|rrr}
\otr & 1 & 2 & 3 \\ \hline
1 & 3 & 3 & 3 \\
2 & 1 & 1 & 1 \\
3 & 2 & 2 & 2 \\
\end{array}
\]
We compute via \texttt{python} that $X$ has biquandle arrow weights over
$\mathbb{Z}_6$ including
\[\scalebox{0.95}{$
b_1=
\left[\begin{array}{ccc}
\left[\begin{array}{rrr}0& 2& 1 \\ 1& 0& 2 \\ 2& 1& 0 \end{array}\right] &
\left[\begin{array}{rrr}2& 0& 3 \\ 3& 2& 0 \\ 0& 3& 2 \end{array}\right] &
\left[\begin{array}{rrr}1& 3& 0 \\ 0& 1& 3 \\ 3& 0& 1 \end{array}\right] \\
 & & \\
\left[\begin{array}{rrr}1& 3& 0 \\ 0& 1& 3 \\ 3& 0& 1 \end{array}\right] &
\left[\begin{array}{rrr}0& 2& 1 \\ 1& 0& 2 \\ 2& 1& 0 \end{array}\right] &
\left[\begin{array}{rrr}2& 0& 3 \\ 3& 2& 0 \\ 0& 3& 2 \end{array}\right] \\
 & & \\
\left[\begin{array}{rrr}2& 0& 3 \\ 3& 2& 0 \\ 0& 3& 2\end{array}\right] &
\left[\begin{array}{rrr}1& 3& 0 \\ 0& 1& 3 \\ 3& 0& 1 \end{array}\right] &
\left[\begin{array}{rrr}0& 2& 1 \\ 1& 0& 2 \\ 2& 1& 0\end{array}\right]
\end{array}\right],\
b_2=
\left[\begin{array}{ccc}
\left[\begin{array}{rrr}0& 3& 3 \\ 3& 0& 3 \\ 3& 3& 0\end{array}\right] &
\left[\begin{array}{rrr}3& 0& 0 \\ 0& 3& 0 \\ 0& 0& 3\end{array}\right] &
\left[\begin{array}{rrr}3& 0& 0 \\ 0& 3& 0 \\ 0& 0& 3\end{array}\right] \\
 & & \\
\left[\begin{array}{rrr}3& 0& 0 \\ 0& 3& 0 \\ 0& 0& 3\end{array}\right] &
\left[\begin{array}{rrr}0& 3& 3 \\ 3& 0& 3 \\ 3& 3& 0\end{array}\right] &
\left[\begin{array}{rrr}3& 0& 0 \\ 0& 3& 0 \\ 0& 0& 3\end{array}\right] \\
 & & \\
\left[\begin{array}{rrr}3& 0& 0 \\ 0& 3& 0 \\ 0& 0& 3\end{array}\right] &
\left[\begin{array}{rrr}3& 0& 0 \\ 0& 3& 0 \\ 0& 3& 3\end{array}\right] &
\left[\begin{array}{rrr}0& 3& 3 \\ 3& 0& 3 \\ 3& 3& 0\end{array}\right] 
\end{array}\right]$}
\]
and that the maps $\sigma_1=[3,1,2]$ and $\sigma_2=[2,3,1]$ are endomorphisms 
of $X$. 

Then the virtual knots $K_1$ and $K_2$ below have the biquandle arrow 
weight quiver representations
\[\includegraphics{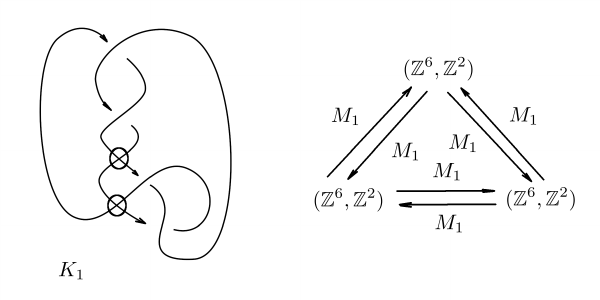}\]
\[\includegraphics{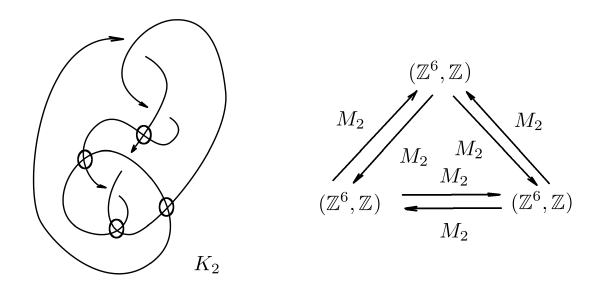}\]
where 
\[
M_1=\left[\begin{array}{rrrrrr}
1 & 0 & 0 & 0 & 0 & 0\\
0 & 0 & 0 & 0 & 0 & 0\\
0 & 0 & 0 & 0 & 0 & 0\\
0 & 0 & 0 & 1 & 0 & 0\\
0 & 0 & 0 & 0 & 0 & 0\\
0 & 0 & 0 & 0 & 0 & 0\\
\end{array}\right]\ \mathrm{and}\
M_2=\left[\begin{array}{rrrrrr}
2 & 0 & 0 & 0 & 0 & 0\\
0 & 0 & 0 & 0 & 0 & 0\\
0 & 0 & 0 & 0 & 0 & 0\\
0 & 0 & 0 & 0 & 0 & 0\\
0 & 0 & 0 & 0 & 0 & 0\\
0 & 0 & 0 & 0 & 0 & 0\\
\end{array}\right].
\]
The unenhanced biquandle coloring quivers are the 
same but their arrow weight representations are different, showing that
$\mathcal{QR}_{\vec{v}}$ is a proper enhancement of both the biquandle
counting invariant and the coloring quiver. 
\end{example}

Many important constructions in knot theory and elsewhere start with
quiver representations; for example, the Khovanov cube at the heart
of Khovanov homology is a quiver representation. Following \cite{N24},
we define four new infinite families of invariants of classical and
virtual knots associated to data vectors $(X,A,B,k,S)$ which can be
regarded as decategorifications of the biquandle arrow weight quiver
representation invariant.

To define our invariants we must first recall a definition from \cite{N24}:
given a matrix $M\in M_{m,n}(A)$ with entries $m_{j,k}$ in row $j$ column
$k$, the \textit{matrix polynomial} of $M$, denoted $mp(M)$, is
\[mp(M)=\sum_{j=1}^{m}\sum_{k=1}^n m_{jk}x^jy^k.\]
That is, the matrix polynomial is a two-variable polynomial encoding the 
entries of the matrix as coefficients.

\begin{definition}
Let $X$ be a finite biquandle, $A$ an abelian group, $B$ a biquandle arrow
weight with $A$ coefficients, $k$ a choice of commutative initial coefficient
ring and $S\subset \mathrm{Hom}(X,X)$ a set of biquandle endomorphisms. 
Then for any oriented classical or virtual knot or link $K$ represented by 
a Gauss diagram $D$, let $E$ be the set of edges, let $M(e)$ the matrix 
associated to the edge $e$ and let $P$ be the set of
maximal (i.e., not a subpath of any other path) non-repeating (i.e., does not 
contain the same edge more than once) paths in the biquandle arrow weight
quiver representation of $K$ associated to the data vector 
$\vec{v}=(X,A,B,k,S)$. Then we define:
\begin{itemize}
\item The \textit{biquandle arrow weight edge characteristic polynomial}
is the sum of the characteristic polynomials of all of the matrices
in the quiver representation, i.e.
\[\Phi_{\vec{v}}^{E,\mathcal{\chi}}(K)=\sum_{e\in E}\chi(M(e))\]
and
\item The \textit{biquandle arrow weight edge matrix polynomial}
is the sum of the matrix polynomials of all of the matrices
in the quiver representation, i.e. 
\[\Phi_{\vec{v}}^{E,M}(K)=\sum_{e\in E} mp(M(e)).\]
\end{itemize}
Now for each path $p=(e_1,\dots, e_n)\in P$, let $M(p)=M(e_n)\dots M(e_1)$ 
be the product along $p$ of the matrices associated to the edges comprising
$p$ (note the reversed order since we think of matrix product as right-to-left)
and let $|p|=n$ be the number of edges in the path $p$. Then we define
\begin{itemize}
\item The \textit{biquandle arrow weight maximal path characteristic polynomial}
is the sum of the characteristic polynomials of all of the matrices
accumulated along the maximal non-repeating paths in the quiver 
representation times a variable encoding the length of the paths, i.e.
\[\Phi_{\vec{v}}^{MP,\mathcal{\chi}}(K)=\sum_{p\in P}\chi(M(p))s^{|p|}\]
and
\item The \textit{biquandle arrow weight maximal path matrix polynomial}
is the sum of the matrix polynomials of all of the matrices
accumulated along the maximal non-repeating paths in the quiver 
representation times a variable encoding the length of the paths, i.e.
\[\Phi_{\vec{v}}^{MP,M}(K)=\sum_{p\in P} mp(M(p))z^{|p|}.\]
\end{itemize}
\end{definition}

\begin{example}
The virtual knots in Example \ref{ex:2} have the following values of the
four decategorification polynomials:

\begin{table}[htbp]
\centering
\small
\begin{tabular}{c|cccc}
\hline
$K$
&
$\Phi_{\bar v}^{E,X}(K)$
&
$\Phi_{\bar v}^{E,M}(K)$
&
$\Phi_{\bar v}^{M,P,X}(K)$
&
$\Phi_{\bar v}^{M,P,M}(K)$
\\
\hline

$K_1$
&
$6t^{6}-12t^{5}+6t^{4}$
&
$6x^{3}y^{3}+6$
&
\makecell{$18s^{6}t^{6}-36s^{6}t^{5}+18s^{6}t^{4}$\\
$+6s^{4}t^{6}-12s^{4}t^{5}+6s^{4}t^{4}$}
&
\makecell{$18z^{6}(x^{3}y^{3}+1)$\\
$+6z^{4}(x^{3}y^{3}+1)$}
\\

$K_2$
&
$6t^{6}-12t^{5}$
&
$12$
&
\makecell{$18s^{6}t^{6}-1152s^{6}t^{5}$\\
$+6s^{4}t^{6}-96s^{4}t^{5}$}
&
\makecell{$1152z^{6}$\\
$+96z^{4}$}
\\

\hline
\end{tabular}
\end{table}
\end{example}

\begin{example}
We computed the values of the edge polynomial invariants for a choice of 
orientation for each prime virtual knots with up to four crossings as found at \cite{KA} 
using our \texttt{python} code for the data vector $(X,A,B,k,S)$ where
$X$ is the biquandle given by the operation tables
\[
\begin{array}{r|rrr}
\utr & 1 & 2 & 3 \\ \hline
1 & 2 & 2 & 2 \\
2 & 3 & 3 & 3 \\
3 & 1 & 1 & 1
\end{array}
\quad \begin{array}{r|rrr}
\otr & 1 & 2 & 3 \\ \hline
1 & 2 & 2 & 2 \\
2 & 3 & 3 & 3 \\
3 & 1 & 1 & 1
\end{array},
\]
$A=\mathbb{Z}_8$, $B=\{\beta_1,\beta_2,\beta_3\}$ is given by
\[\beta_1=
\left[\begin{array}{rrr}
\left[\begin{array}{rrr} 0& 1& 3 \\ 3& 0& 1 \\ 1& 3& 0\end{array}\right] &  
\left[\begin{array}{rrr} 1& 0& 4 \\ 4& 1& 0 \\ 0& 4& 1\end{array}\right]&  
\left[\begin{array}{rrr} 3& 4& 0 \\ 0& 3& 4 \\ 4& 0& 3\end{array}\right] \\
\left[\begin{array}{rrr} 3& 4& 0 \\ 0& 3& 4 \\ 4& 0& 3\end{array}\right]&  
\left[\begin{array}{rrr} 0& 1& 3 \\ 3& 0& 1 \\ 1& 3& 0\end{array}\right]&   
\left[\begin{array}{rrr} 1& 0& 4 \\ 4& 1& 0 \\ 0& 4& 1\end{array}\right] \\
\left[\begin{array}{rrr} 1& 0& 4 \\ 4& 1& 0 \\ 0& 4& 1\end{array}\right]&   
\left[\begin{array}{rrr} 3& 4& 0 \\ 0& 3& 4 \\ 4& 0& 3\end{array}\right]&  
\left[\begin{array}{rrr} 0& 1& 3 \\ 3& 0& 1 \\ 1& 3& 0\end{array}\right]\end{array}\right],
\beta_2=
\left[\begin{array}{rrr} 
\left[\begin{array}{rrr} 0& 2& 2 \\ 2& 0& 2 \\ 2& 2& 0 \end{array}\right]&   
\left[\begin{array}{rrr} 2& 0& 4 \\ 4& 2& 0 \\ 0& 4& 2 \end{array}\right]&  
\left[\begin{array}{rrr} 2& 4& 0 \\ 0& 2& 4 \\ 4& 0& 2 \end{array}\right]  \\
\left[\begin{array}{rrr} 2& 4& 0 \\ 0& 2& 4 \\ 4& 0& 2 \end{array}\right]&  
\left[\begin{array}{rrr} 0& 2& 2 \\ 2& 0& 2 \\ 2& 2& 0 \end{array}\right]&  
\left[\begin{array}{rrr} 2& 0& 4 \\ 4& 2& 0 \\ 0& 4& 2 \end{array}\right] \\
\left[\begin{array}{rrr} 2& 0& 4 \\ 4& 2& 0 \\ 0& 4& 2 \end{array}\right]&  
\left[\begin{array}{rrr} 2& 4& 0 \\ 0& 2& 4 \\ 4& 0& 2 \end{array}\right]&  
\left[\begin{array}{rrr} 0& 2& 2 \\ 2& 0& 2 \\ 2& 2& 0 \end{array}\right]\end{array}\right],
\] and
\[
\beta_3=
\left[\begin{array}{rrr}
\left[\begin{array}{rrr} 0& 6& 6 \\ 6& 0& 6 \\ 6& 6& 0 \end{array}\right]&  
\left[\begin{array}{rrr} 6& 0& 4 \\ 4& 6& 0 \\ 0& 4& 6 \end{array}\right]&  
\left[\begin{array}{rrr} 6& 4& 0 \\ 0& 6& 4 \\ 4& 0& 6 \end{array}\right] \\
\left[\begin{array}{rrr} 6& 4& 0 \\ 0& 6& 4 \\ 4& 0& 6 \end{array}\right]&  
\left[\begin{array}{rrr} 0& 6& 6 \\ 6& 0& 6 \\ 6& 6& 0 \end{array}\right]&  
\left[\begin{array}{rrr} 6& 0& 4 \\ 4& 6& 0 \\ 0& 4& 6 \end{array}\right] \\
\left[\begin{array}{rrr} 6& 0& 4 \\ 4& 6& 0 \\ 0& 4& 6 \end{array}\right]&  
\left[\begin{array}{rrr} 6& 4& 0 \\ 0& 6& 4 \\ 4& 0& 6 \end{array}\right]&  
\left[\begin{array}{rrr} 0& 6& 6 \\ 6& 0& 6 \\ 6& 6& 0 \end{array}\right]\end{array}\right]\]
and $S=\{[2,3,1],[3,1,2]\}$.

The results are in the table.

\begin{longtable}{p{0.62\textwidth}|p{0.30\textwidth}}
\hline
$K$ &
$(\Phi_{\bar v}^{E,\chi}(K),\ \Phi_{\bar v}^{E,M}(K))$
\\
\hline
\endfirsthead

\hline
$K$ &
$(\Phi_{\bar v}^{E,\chi}(K),\Phi_{\bar v}^{E,M}(K))$
\\
\hline
\endhead

2.1, 3.1, 3.2, 4.4, 4.5, 4.9, 4.14, 4.27, 4.30, 4.37, 4.44, 4.48,
4.52, 4.54, 4.60, 4.64, 4.69, 4.74, 4.82, 4.84, 4.91, 4.94, 4.102
&
$(6t^{8}-18t^{7},\,18x^{4}y^{4})$
\\

3.3, 3.4, 3.5, 3.6, 3.7, 4.1, 4.2, 4.3, 4.6, 4.7, 4.8, 4.12, 4.16,
4.21, 4.25, 4.31, 4.36, 4.43, 4.46, 4.53, 4.55, 4.56, 4.59, 4.65,
4.71, 4.73, 4.75, 4.76, 4.77, 4.85, 4.86, 4.89, 4.90, 4.92, 4.95,
4.96, 4.98, 4.99, 4.100, 4.101, 4.104, 4.105, 4.106, 4.107, 4.108
&
$(6t^{8}-18t^{7},\,18)$
\\

4.10, 4.18, 4.19, 4.33, 4.35, 4.40, 4.50, 4.57, 4.70
&
$(6t^{8}-18t^{7}+12t^{6},\,6x^{2}y^{2}+12)$
\\

4.11, 4.63, 4.68
&
$(6t^{8}-18t^{7}+12t^{6},\,12x^{4}y^{4}+6x^{2}y^{2})$
\\

4.13, 4.15, 4.20, 4.29, 4.34, 4.38, 4.41, 4.49, 4.51, 4.58, 4.72
&
$(6t^{8}-18t^{7}+12t^{6},\,6x^{6}y^{6}+12x^{4}y^{4})$
\\

4.17, 4.23, 4.32, 4.61, 4.67
&
$(6t^{8}-18t^{7}+12t^{6},\,6x^{6}y^{6}+12)$
\\

4.22, 4.62, 4.83, 4.87, 4.88
&
$(6t^{8}-18t^{7}+18t^{6}-6t^{5},\,
6x^{6}y^{6}+6x^{3}y^{3}+6x^{2}y^{2})$
\\

4.24, 4.26, 4.42, 4.47, 4.80
&
$(6t^{8}-18t^{7}+18t^{6}-6t^{5},\,
6x^{6}y^{6}+6x^{2}y^{2}+6xy)$
\\

4.28, 4.39, 4.45, 4.78, 4.81, 4.97, 4.103
&
$(6t^{8}-18t^{7}+18t^{6}-6t^{5},\,
6x^{6}y^{6}+6x^{5}y^{5}+6x^{2}y^{2})$
\\

4.66, 4.79, 4.93
&
$(6t^{8}-18t^{7}+18t^{6}-6t^{5},\,
6x^{7}y^{7}+6x^{6}y^{6}+6x^{2}y^{2})$
\\
\hline

\end{longtable}

\end{example}

\begin{example}
We computed the values of the maximal path polynomial invariants 
for the prime virtual knots with up to four crossings as found at \cite{KA} 
using our \texttt{python} code for the data vector 
\[\vec{v}=(X,\mathbb{Z}_{12},\{b_1,b_2,b_3\},\mathbb{Z},\{[1,2],[2,1]\})\]
where 
\[b_1=
\left[\begin{array}{rr}
\left[\begin{array}{rr} 0 & 9 \\ 9 & 0 \end{array}\right] & 
\left[\begin{array}{rr} 9 & 0 \\ 6 & 3 \end{array}\right] \\
\left[\begin{array}{rr} 9 & 6 \\ 0 & 3 \end{array}\right] &
\left[\begin{array}{rr} 0 & 3 \\ 3 & 0 \end{array}\right] 
\end{array}\right],
b_2=
\left[\begin{array}{rr}
\left[\begin{array}{rr} 0 & 0 \\ 6 & 6 \end{array}\right] & 
\left[\begin{array}{rr} 0 & 0 \\ 6 & 6 \end{array}\right] \\
\left[\begin{array}{rr} 6 & 6 \\ 0 & 0 \end{array}\right] &
\left[\begin{array}{rr} 6 & 6 \\ 0 & 0\end{array}\right] 
\end{array}\right],
b_3=
\left[\begin{array}{rr}
\left[\begin{array}{rr} 0 & 6 \\ 0 & 6 \end{array}\right] & 
\left[\begin{array}{rr} 6 & 0 \\ 6 & 0 \end{array}\right] \\
\left[\begin{array}{rr} 0 & 6 \\ 0 & 6\end{array}\right] &
\left[\begin{array}{rr} 6 & 0 \\ 6 & 0 \end{array}\right] 
\end{array}\right]
\]
and $X$ is the biquandle given by the operation tables
\[
\begin{array}{r|rr}
\utr & 1 & 2 \\ \hline
1 & 2 & 2 \\
2 & 1 & 1
\end{array}
\quad
\begin{array}{r|rr}
\otr & 1 & 2 \\ \hline
1 & 2 & 2 \\
2 & 1 & 1
\end{array}.
\] 
The results are in the tables.

\[\begin{array}{l|l}
\hline
 \Phi_{\vec{v}}^{MP,\chi}(K) & K \\ \hline
2.1, 3.1, 3.2, 3.3, 3.4, 3.5, 3.6, 3.7, 4.1, 4.2, 4.3, 4.4, 4.5, 4.6, 4.7, 4.8, 4.9, 4.10, & 4s^4t^{12} - 324s^4t^{11}  \\ 
4.11, 4.12, 4.13, 4.14, 4.15, 4.16, 4.17, 4.18, 4.19, 4.20, 4.21, 4.22, 4.23, 4.24, & \\ 4.25, 4.26, 4.27, 4.28, 4.43, 4.44, 4.45, 4.46, 4.47, 4.48, 4.49, 4.50, 4.51, 4.52, &\\ 4.53, 4.54, 4.55, 4.56, 4.73, 4.74, 4.75, 4.76, 4.77, 4.78, 4.79, 4.80, 4.81, 4.82, & \\   4.83, 4.84, 4.85, 4.86, 4.87, 4.88, 4.89, 4.90, 4.91, 4.92, 4.93, 4.94, 4.95, 4.96, & \\  4.97, 4.98, 4.99, 4.100, 4.101, 4.102, 4.103, 4.104, 4.105, 4.106, 4.107, 4.108 & \\
4.29, 4.30, 4.31, 4.32, 4.33, 4.34, 4.35, 4.36, 4.37, 4.38, 4.39, 4.40, 4.41, 4.42, & 4s^4t^{12}-68s^4t^{11} + 64s^4t^{10}  \\  4.57, 4.58, 4.59, 4.60, 4.61, 4.62, 4.63, 4.64, 4.65, 4.66, 4.67, 4.68, 4.69, 4.70, \\ 4.71, 4.72  & \\ \hline
\end{array}
\]
\[
\begin{array}{l|l}
\hline
\Phi_{\vec{v}}^{MP,M}(K) & K \\ \hline
2.1, 3.2, 3.3, 3.4, 4.4, 4.5, 4.9, 4.11, 4.14, 4.15, 4.18, 4.20, 4.22, 4.27, 4.28, 4.44, 4.45, 4.48, & 324x^6y^6z^4 \\  4.49, 4.52, 4.54, 4.74, 4.78, 4.81, 4.82, 4.83, 4.84, 4.87, 4.88, 4.92, 4.94, 4.95, 4.101, 4.104 & \\
3.1, 3.5, 3.6, 3.7, 4.1, 4.2, 4.3, 4.6, 4.7, 4.8, 4.10, 4.12, 4.13, 4.16, 4.17, 4.19, 4.21, 4.23, 4.24, & 324z^4  \\  4.25, 4.26, 4.43, 4.46, 4.47, 4.50, 4.51, 4.53, 4.55, 4.56, 4.73, 4.75, 4.76, 4.77, 4.79, 4.80, &\\  4.85, 4.86, 4.89, 4.90, 4.91, 4.93, 4.96, 4.97, 4.98, 4.99, 4.100, 4.102, 4.103, 4.105, 4.106, & \\  4.107, 4.108 \\
4.29, 4.30, 4.33, 4.34, 4.37, 4.38, 4.39, 4.40, 4.60, 4.61, 4.62, 4.63, 4.64, 4.69 & 4z^4(16x^6y^6 + 1) \\
4.31, 4.32, 4.35, 4.36, 4.41, 4.42, 4.57, 4.58, 4.59, 4.65, 4.66, 4.67, 4.68, 4.70, 4.71, 4.72 & 4z^4(x^6y^6 + 16) \\
\hline
\end{array}
\]

\end{example}

\section{\Large\textbf{Questions}}\label{Q}

We conclude with some questions for future research. 

First, we note that the examples in this paper are small toy examples 
using small finite biquandles and small finite rings, meant
to show the computation of the invariants and should not be taken as
representative of the power of this infinite family of invariants.
Examples with larger biquandles, more endomorphisms and larger abelian 
groups are expect to yield stronger invariants.

In light of this, faster computation algorithms for finding biquandle 
arrow weights for larger choices of biquandle and abelian group are
of great interest.

What are the geometric meanings of these invariants? How do they relate 
to other quiver representations? What is the algebraic structure of the set of
such representations for a fixed knot, biquandle etc.?

\bibliographystyle{abbrv}
\bibliography{sam-migiwa}

\bigskip

\noindent
\textsc{Department of Mathematical Sciences \\
Claremont McKenna College \\
850 Columbia Ave. \\
Claremont, CA 91711} 

\

\noindent
\textsc{ \\
Department of Materials Science and Engineering,\\
College of Engineering, Shibaura Institute of Technology,\\
307 Fukasaku, Minuma-ku, Saitama-shi, Saitama, 337-8570, Japan
}

\end{document}